\documentclass[a4paper,12pt,reqno]{amsart}

\usepackage{amsthm}
\usepackage{amsmath,amssymb,latexsym,amsfonts,mathrsfs}
\usepackage{bm}

\usepackage{color}
\usepackage{comment}
\usepackage{mathtools}
\usepackage{fancyhdr}
\usepackage[top=30truemm,bottom=30truemm,left=20truemm,right=20truemm]{geometry}
\usepackage{enumitem}
\usepackage{enumerate}

\newcommand{\R}{\mathbb{R}}
\newcommand{\C}{\mathbb{C}}
\newcommand{\N}{\mathbb{N}}

\newcommand{\SCR}[1]{{\mathscr #1}}

\newcommand{\CAL}[1]{{\mathcal #1}}

\theoremstyle{plain}
\newtheorem{Thm}{{\bf Theorem}}[section]

\newtheorem{Lem}[Thm]{{\bf Lemma}}
\newtheorem{Prop}[Thm]{{\bf Proposition}}

\theoremstyle{definition}
\newtheorem{Def}[Thm]{{\bf Definition}}
\newtheorem{Rem}[Thm]{{\bf Remark}}

\newcounter{Exami}

\allowdisplaybreaks[2]

\makeatletter
 \def\address#1#2{\begingroup
 \noindent\parbox[t]{7.8cm}{%
 \small{\scshape\ignorespaces#1}\par\vskip0ex
 \noindent\small{\itshape E-mail}%
 \/: #2\par\vskip4ex}\hfill%
 \endgroup}%
\makeatother

\begin{document}
\fontencoding{T1}\selectfont
\title[Asymptotic behavior of DNLS]{Asymptotic behavior for the damped Schr\"{o}dinger equation with nonlinear dissipation.}
\author{Takahisa Inui, Shun Takizawa}
\begin{abstract}
We consider large time asymptotics of solutions to the damped Schr\"{o}dinger equation with the nonlinear dissipation in the mass-subcritical case. 
We prove that the optimal $L^2$-decay rate of the nonlinear solution coincides with that of the corresponding linear solution even in the presence of nonlinear dissipation. 
Moreover, we give the optimal convergence rate for scattering for any power in the mass-subcritical regime. 
\end{abstract}
\maketitle
\section{Introduction} 
In this paper, we study the following damped nonlinear Schr\"{o}dinger equation:
\begin{align}\label{CP1}
&i\partial_{t}u+\Delta u+ia u=\lambda |u|^{p-1} u, \hspace{3mm}(t,x)\in [0,\infty) \times \R^n
\end{align}
with the initial data $u|_{t=0}=u_0$, where $a>0$, $p>1$ and $\lambda\in \C$. The damped nonlinear Schr\"{o}dinger equation
appears in various areas of nonlinear optics, plasma physics, and fluid
dynamics, and has been studied from both mathematical and physical viewpoints. 
See \cite{Goldman, MTsutsumi1, MTsutsumi2, Fibich, Ohta-Todorova, Mohamad, Hamouda}.
A feature of the damped Schr\"{o}dinger equation is that the solution to (\ref{CP1}) obeys 
\begin{align} \label{decay}
\| u(t)\|_{L^2}=e^{-at} \|u(0)\|_{L^2}
\end{align} 
for all $t\in[0,\infty)$ when $\lambda\in \R$.
Thus the damping term $iau$ of (\ref{CP1}) causes the mass of solutions to decay.

There are numerous works on large time behavior for the standard nonlinear Schr\"{o}dinger equation, corresponding to (\ref{CP1}) with $\operatorname{Im}\lambda=0$ and $a=0$.
For results on $L^2$-scattering and $H^1$-scattering, we refer to \cite{Tsutsumi-Yajima, Ozawa, Barab, Hayashi-Naumkin, Kita-Wada, Kita, Kita-Ozawa, Hayashi-Kawahara} and \cite{YTsutsumi1, Ginibre-Velo, Nakanishi, Bourgain, Dodson, BGTV} and references therein, respectively.

Next, we review  previous results in the case with $\operatorname{Im}\lambda<0$ and $a=0$.
The condition $\operatorname{Im}\lambda<0$ makes the $L^2$-norm of the solution decrease as shown by the following identity:
\begin{align*}
\|u(t)\|_{L^2}^{2}=\|u(0)\|_{L^2}^{2}+2 \operatorname{Im}\lambda \int_{0}^{t} \|u(s)\|_{L^{p+1}}^{p+1}
ds
\end{align*}
for the solution $u(t)$ to (\ref{CP1}) with $a=0$.
 Kita--Shimomura \cite{Kita-Shimomura} showed that $\displaystyle\lim_{t\to\infty} \|u(t)\|_{L^2}=0$ if $n=1$ and $1<p\leq3$.  
After that, Hayashi--Li--P. Naumkin \cite{Hayashi-Li-Naumkin1} proved
$\displaystyle\lim_{t\to\infty} \|u(t)\|_{L^2}=0$ provided $1<p< 1+2/n$.
For the exponent $p=1+2/n$, we refer to Hayashi--Li--P. Naumkin \cite{Hayashi-Li-Naumkin2} and  Cazenave--I. Naumkin \cite{Cazenave-Naumkin}. 
On the other hand, Sato \cite{Sato} proved that the solution of (\ref{CP1}) with the initial condition $u|_{t=0}=u_{0}\ne0$ has a positive $L^2$-lower bound when $p>1 +2/n$. Therefore we see that the nonlinear dissipation $\lambda |u|^{p-1}u$ affects large time behavior of the solution to (\ref{CP1}) with $a=0$ only when $p\leq 1+2/n$. 

In addition, for the case with $\lambda\in\R$ and $a>0$ we have the exponential decay (\ref{decay}).
We clarify in Theorem \ref{thm2} below that only the linear dissipation contributes to the decay of the solution to (\ref{CP1}) for  $1<p<1+4/n$ when both linear and nonlinear dissipation are present.

By setting $v(t)=e^{at}u(t)$, we can reduce (\ref{CP1}) with the initial condition $u|_{t=0}=u_{0}$ to the following Cauchy problem:
\begin{align} \label{CP2}
\begin{cases}
&i\partial_{t}v+\Delta v=\lambda e^{-a(p-1)t} |v|^{p-1} v, \hspace{3mm}(t,x)\in [0,\infty) \times \R^n,
\\&v|_{t=0}=v_{0}.
\end{cases}
\end{align} 
Then we note $u_0=v_0$.

In order to address large time asymptotics for (\ref{CP2}), we first give the global well-posedness result.

\begin{Prop}\label{prop1}
Let $v_0\in L^2(\R^n)$ and $1<p<1+4/n$. Then there exists a unique global solution $v\in C([0,\infty); L^2)\cap L^{q}([0,\infty); L^r)$ of (\ref{CP2}) with $\operatorname{Im}\lambda\leq0$, where $(q,r)=(\frac{4(p+1)}{n(p-1)}, p+1)$. Moreover, this global solution satisfies
\begin{align}\label{eq3}
 \|v(t)\|_{L^2}^2=\|v_0\|_{L^2}^2 +2 \operatorname{Im}\lambda \int_{0}^{t} e^{-a(p-1)s} \|v(s)\|_{L^{p+1}}^{p+1}
ds
\end{align}
and there exists a scattering state $\phi\in L^2(\R^n)$ such that
\begin{align*}
\lim_{t\to\infty} e^{(p-1)at} \left( \|v(t)-e^{it\Delta}\phi\|_{L^2}+ \|v(s)-e^{is\Delta}\phi\|_{L^q([t,\infty);L^r)}\right)=0.
\end{align*}
\end{Prop}
This proof is given in Appendix A below.
 
If $\operatorname{Im} \lambda=0$, then it immediately follows from (\ref{eq3}) that the conservation law $\|v(t)\|_{L^2}=\|v_0\|_{L^2}$ holds. On the other hand, in the case $\operatorname{Im} \lambda<0$, we see that $\|v(t)\|_{L^2}$ is monotone decreasing similar to the case $a=0$. 
We reveal that the solution does not vanish as $t\to\infty$ for all $1<p<1+4/n$ even if $\operatorname{Im} \lambda<0$ in the following theorem. 
\begin{Thm} \label{thm2}
Let $v$ be the solution of (\ref{CP2}) constructed in Proposition \ref{prop1} with $\operatorname{Im}\lambda<0$ and the initial data $v_0\ne 0$.

Then there exists a constant $C>0$ such that 
\begin{align*}
\|v(t)\|_{L^2}>C
\end{align*}
for all $t\in [0,\infty)$. 
\end{Thm}
The statement of Theorem \ref{thm2} means that the solution of (\ref{CP1}) has the lower bound: $\|u(t)\|_{L^2}\geq C e^{-at}$ with some constant $C>0$. This implies that the nonlinear dissipation does not affect the decay of the solution when the linear damping presents for all $1<p<1+4/n$. 
To the best of our knowledge, Theorem \ref{thm2} is the first result on the decay estimate for the solution to the nonlinear Schr\"{o}dinger equation with
both the linear damping and the nonlinear dissipation.

Moreover,  in Theorem \ref{thm1} below,  we give the asymptotic behavior of the solution to (\ref{CP1}) more details.
For the case $\lambda\in \R$, large time asymptotic behavior for (\ref{CP1}) has been studied by Inui \cite{Inui} and Aloui--Jbari--Tayachi \cite{Aloui}.
They proved that if  there exists a scattering state $\phi\in H^1$ such that $\|e^{at}u(t)-e^{it\Delta}\phi\|_{H^1}\to0$ as $t\to\infty$  for the solution $u\in C([0,\infty); H^1)$ to (\ref{CP1}) with $u(0)\in H^1$, then its convergence rate is obtained as follows:
\begin{align} \label{Inq2}
\|u(t)-e^{-at}e^{it\Delta}\phi\|_{H^1}=o\left(e^{-apt}\right)
\end{align} 
with $1<p<1+\frac{4}{n-2}$, where we regard $1+\frac{4}{n-2}$ as $\infty$ when $n=1, 2$. In recent work, Takagi--Takizawa \cite{Takagi-Takizawa} proved that the optimal convergence rate of $L^2$-scattering is $t^{-n(p-1)/2} e^{-apt}$ when the scattering state belongs to the modulation space $M^{1,1}(\R^n)$ and $p \in 2\N+1$. 

In the present paper, we prove that optimal convergence rate of $L^2$-scattering is $t^{-n(p-1)/2} e^{-apt}$ when the scattering state belongs to $L^{\frac{2p}{2p-1}}(\R^n)$ with $1<p<1+4/n$ in Theorem \ref{thm1} below.

\begin{Thm}\label{thm1}
Let $v$ be the solution of (\ref{CP2}) stated in Proposition \ref{prop1} with $v_0\ne 0$. If the scattering state $\phi$  belongs to not only $L^2(\R^n)$ but also $L^{\frac{2p}{2p-1}}(\R^n)$, then there exist constants $C_2>C_1>0$ and large $T>0$ such that 
\begin{align*}
C_1 \left(t^{-n/2}e^{-at}\right)^{p-1}\leq \|v(t)-e^{it\Delta}\phi\|_{L^2} \leq C_2 \left(t^{-n/2}e^{-at}\right)^{p-1}
\end{align*}
for all $t>T$. 
\end{Thm}
 Rewriting the statement of Theorem \ref{thm1} by the transformation $v(t)=e^{at}u(t)$, we have 
\begin{align*}
C_1 t^{-n(p-1)/2} e^{-apt} \leq \|u(t)-e^{-at} e^{it\Delta}\phi\|_{L^2} \leq C_2 t^{-n(p-1)/2}e^{-apt}
\end{align*}
for the solution $u(t)$ to (\ref{CP1}).

We note that the scattering state $\phi$ in Theorem \ref{thm1} is nonzero. Indeed,
if $\operatorname{Im}\lambda=0$, then (\ref{eq3}) leads to
\begin{align*}
\|\phi\|_{L^2}=\lim_{t\to\infty}\|e^{-it\Delta}v(t)\|_{L^2}=\lim_{t\to\infty}\|v(t)\|_{L^2}=\|v_0\|_{L^2}
\end{align*}
 and if $\operatorname{Im}\lambda<0$, then Theorem \ref{thm2} yields
\begin{align*}
\|\phi\|_{L^2}=\lim_{t\to\infty}\|e^{-it\Delta}v(t)\|_{L^2}=\lim_{t\to\infty} \|v(t)\|_{L^2}\geq C,
\end{align*}
where the constant $C>0$ is as in Theorem \ref{thm2}.

\begin{Rem}
The assumption $\phi \in L^{\frac{2p}{2p-1}}(\R^n)$ in Theorem \ref{thm1} is satisfied for some large initial data if $\operatorname{Im}\lambda=0$. Namely, for any $M>0$ there exists an initial data $u(0)\in L^2$ such that $\|u(0)\|_{L^2}>M$ and a scattering state $\phi$ belongs to $L^{\frac{2p}{2p-1}}(\R^n)$ for $\lambda\in \R$. We present the result on the final state problem (Proposition \ref{prop2}) in Appendix B.
In the case $\operatorname{Im}\lambda<0$, the problem is still open because the nonlinear dissipation amplifies the mass of the solution when the equation is solved backward in time for the final state problem. Indeed, it is known that $L^2$-solution blows up backward in finite time for small data in the case where $a=0$ and $\operatorname{Im}\lambda <0$ by Cazenave--Martel--Zhao \cite{CMZ} and Kita \cite{Kita1}.
\end{Rem}
\begin{Rem}
Compared to the result of Takagi--Takizawa \cite{Takagi-Takizawa}, we improve the following:
\begin{itemize}
\item  We can deal with full range $p\in(1, 1+4/n)$, while the previous work \cite{Takagi-Takizawa} treat  the case $p \in 2 \N+1$. 
\item We address widely the function space since $M^{1,1}\subset \left( L^1 \cap L^{\infty} \right)  \subset \left( L^2 \cap L^{\frac{2p}{2p-1}} \right) $ (see e.g. \cite[Proposition 2.3]{Takagi-Takizawa}). 
\end{itemize}
\end{Rem}

The paper is organized as follows. In Section 2, we collect several auxiliary estimates. Sections 3 and 4  are devoted to the proofs of Theorems \ref{thm2} and \ref{thm1}, respectively. Finally, we give the global well-posedness result (Proposition \ref{prop1}), existence for initial values satisfying the assumption in Theorem \ref{thm1} (Proposition \ref{prop2}) and an alternative proof of Theorem \ref{thm2} in Appendices.

\subsection*{Notation}
We write $N(u)=\lambda|u|^{p-1}u$. 
We say that $u, v$ are solutions to (\ref{CP1}) and (\ref{CP2}) if they satisfy the corresponding integral equations
\begin{align*}
&u(t)=e^{-at}e^{it\Delta}u_0-i\int_{0}^{t}e^{-a(t-s)} e^{i(t-s)\Delta} N(u(s))ds,
\\
&v(t)=e^{it\Delta}v_0-i\int_{0}^{t}e^{-a(p-1)s} e^{i(t-s)\Delta} N(v(s))ds,
\end{align*} 
respectively.
Such solutions are called mild solutions.
We write
\begin{align*}
&\SCR{F}f(\xi)=\widehat{f}(\xi)=(2\pi)^{-n/2}\int_{\R^n}f(x)e^{-ix\cdot\xi}dx,
\\
&\SCR{F}^{-1}f(x)=\check{f}(x)=(2\pi)^{-n/2} \int_{\R^n}f(\xi)e^{ix\cdot\xi}d\xi
\end{align*} 
for the Fourier transform and the inverse Fourier transform of $f$, respectively. 
We define the free Schr\"{o}dinger propagator $e^{it\Delta}$ by
\begin{align} \label{formula}
e^{it\Delta}f(x)=\SCR{F}^{-1} e^{-it|\xi|^2} \SCR{F}f(x)&=(4\pi it)^{-n/2}\int_{\R^n} e^{i\frac{(x-y)^2}{4t}} f(y) dy \notag
\\
&=(4\pi it)^{-n/2} e^{i\frac{|x|^{2}}{4t}} \int_{\R^n} \left( e^{i\frac{|y|^2}{4t}} f(y)\right) e^{iy\cdot \frac{x}{2t}} dy \notag
\\
&=\left( M(t)D(t)\SCR{F} M(t) f \right) (x).
\end{align} 
The most right hand side of the above equality is called MDFM (or Dollard) decomposition formula,
where $M(t) = e^{i|x|^{2}/(4t)}$ is the multiplier and $D(t)$ is the dilation which leaves the $L^2$-norm
invariant by $D(t)f(x)=(2it)^{-n/2}f(x/2t)$. 
We often write $L^q=L^q(\R^n)$ for short.
We often use the notation $X\lesssim Y$ in the proofs if $X\leq CY$ with some constant $C>0$. 
For a Lebesgue exponent $\alpha\in [1,\infty]$, $\alpha'$ denotes its H\"{o}lder conjugate.
%
%
%

\section{Preliminaries}
We introduce some lemmas and known properties which are used in order to prove our results. 
First we have the following estimate for the linear Schr\"{o}dinger evolution operator by (\ref{formula}) and the Hausdorff--Young inequality:
\begin{align} \label{LpLq}
\|e^{it\Delta}f\|_{L^q}\leq (4 \pi |t|)^{-\frac{n}{2}(1-\frac{2}{q})} \|f\|_{L^{q'}}
\end{align}
for $t\ne 0$ and $q\in [2, \infty]$.

We recall the Strichartz estimates for the free Schr\"{o}dinger propagator (see e.g. \cite{Strichartz, Ginibre-Velo1, Yajima4, Keel-Tao}).

\begin{Def}
We call $(q,r)$ an $admissible$ $pair$ if 
\begin{align*}
2\leq q, r\leq \infty, \hspace{3mm}\frac{2}{q}+\frac{n}{r}=\frac{n}{2}, \hspace{3mm} (n,q,r)\ne (2,2,\infty).
\end{align*}
\end{Def}
\begin{Lem}[Strichartz estimates]
Let $(q,r), (\widetilde{q}, \widetilde{r})$ be admissible pairs, $I$ be an interval on $\R$ and $t_0$ belong to closure of $I$. 
Then it holds that
\begin{align*}
\| e^{it\Delta}f\|_{L^{q}(\R; L^r)} &\leq C \|f\|_{L^2},
\\
\left\| \int_{t_0}^{t} e^{i(t-s)\Delta} F(s) ds \right\|_{L^{\widetilde{q}}(I; L^{\widetilde{r}})} &\leq C \|F\|_{L^{q'}(I; L^{r'})},
\end{align*}
where the constant $C>0$ depends only on $n, q$ and $\widetilde{q}$.
\end{Lem}
 
In particular, we often employ the pair $(q,r)=(\frac{4(p+1)}{n(p-1)}, p+1)$ as an admissible pair in this paper.
The following lemma is elementary estimates which are used in the subsequent proofs. 
\begin{Lem} \label{ElemLem}
Let $\alpha\ne0$ and $\beta>0$. There exist constants $C_2>C_1>0$ and $T>0$ such that
\begin{align*}
C_1 t^{\alpha} e^{-\beta t}\leq \int_{t}^{\infty} s^{\alpha} e^{-\beta s} ds \leq C_2 t^{\alpha} e^{-\beta t}
\end{align*}
for all $t>T$.
\end{Lem}
\begin{proof}
The result follows easily from integration by parts.
\end{proof}

The following nonlinear estimate plays an important role and is often used in this paper.
\begin{Lem} \label{Lem2}
Let $(q,r):=(\frac{4(p+1)}{n(p-1)}, p+1)$. 
Then the following nonlinear estimate holds:
\begin{align*}
\|e^{-a(p-1)\tau} N(v(\tau)) \|_{L_{\tau}^{q'}([s,t); L^{r'})} \leq C e^{-a(p-1)s} \|v\|_{L^{q}([s,t); L^r)}^p 
\end{align*}
for $0\leq s<t\leq \infty$ with some constant $C>0$, which depends only on $a,n,\lambda$ and  $p$.
\end{Lem}
\begin{proof}
We put 
\begin{align} \label{def1}
\alpha:=\frac{2(p^2-1)}{4-(n-2)(p-1)}, \hspace{3mm} \beta:=\frac{4(p-1)}{4-n(p-1)}.
\end{align}
Using H\"{o}lder inequalities three times as $\frac{1}{r'}=\frac{1}{r}+\frac{p-1}{p+1}, \hspace{2mm}\frac{1}{q'}=\frac{1}{q}+\frac{p-1}{\alpha}$ and $\frac{1}{\alpha}=\frac{1}{\beta}+\frac{1}{q}$, we have
\begin{align*}
\|e^{-a(p-1)\tau} N(v(\tau)) \|_{L_{\tau}^{q'}([s,t); L^{r'})} &\leq   \Bigl\|e^{-a(p-1)\tau} \||v(\tau)|^{p-1}\|_{L^{\frac{p+1}{p-1}}} \|v(\tau)\|_{L^r} \Bigr\|_{L^{q'}[s,t)}
\\
&=   \Bigl\| \left(e^{-a(p-1)\tau} \|v(\tau)\|_{L^r}^{p-1}\right) \|v(\tau)\|_{L^r} \Bigr\|_{L^{q'}[s,t)}
\\
&\leq   \Bigl\| e^{-a(p-1)\tau} \|v(\tau)\|_{L^r}^{p-1} \Bigr\|_{L^{\frac{\alpha}{p-1}}[s,t)}  \|v\|_{L^{q}([s,t); L^r)} 
\\
&=   \Bigl\| e^{-a(1-1/p)\tau} \|v(\tau)\|_{L^r} \Bigr\|_{L^{\alpha}[s,t)}^{p-1}  \|v\|_{L^{q}([s,t); L^r)} 
\\
&\leq \| e^{-a(1-1/p)\tau} \|_{L^{\beta}[s,t)}^{p-1} \|v\|_{L^{q}([s,t); L^r)}^p
\\
&\lesssim  e^{-a(p-1)s} \|v\|_{L^{q}([s,t); L^r)}^p,
\end{align*}
which completes the proof.
\end{proof}

The following lemma is used in Section 3.
\begin{Lem} \label{lem31}
Let $1<p<1+4/n, (q,r)=(\frac{4(p+1)}{n(p-1)}, p+1)$ and $v: [0,\infty)\times \R^n \to \C$ be a function in $L_{\operatorname{loc}}^{q}([0,\infty); L^r(\R^n))$.
Then it holds that
\begin{align*}
\int_{s}^{t} e^{-a(p-1)\tau} \|v(\tau)\|_{L^{r}}^{r} d\tau \leq \|v\|_{L^{q}([s,t); L^{r})}^{r}
\end{align*}
for $0\leq s\leq t$.
\end{Lem}
\begin{proof}
Noting that $q>r$, we have by the H\"{o}lder inequality 
\begin{align*}
\int_{s}^{t} e^{-a(p-1)\tau} \|v(\tau)\|_{L^{r}}^{r} d\tau &\leq \| e^{-a(p-1)\tau}\|_{L^{\gamma}[s,t)} \Bigl\| \|v(\tau)\|_{L^{r}}^{r} \Bigr\|_{L^{q/r}[s,t)}
\\
&\leq e^{-a(p-1)s} \| v\|_{L^q([s,t); L^r)}^r,
\end{align*}
where $\gamma$ is the constant which satisfies $1=\frac{1}{\gamma}+\frac{r}{q}$.
\end{proof}
\section{Proof of Theorem \ref{thm2}}
We first have 
\begin{align} \label{inq33}
\| v\|_{L^q ([T,\infty); L^r)} \leq C \|v(T)\|_{L^2}
\end{align}
with some constants $C>0$ and $T>0$ by an argument similar to that in STEP 2 in the proof of Proposition \ref{prop1}.
We prove it by contradiction.
Assume that $\displaystyle\lim_{t\to \infty} \| v(t)\|_{L^2}=0$.
Then by the continuity of $v(t)$ in $t$ and $v_0\ne 0$, we possess
\begin{align} \label{inq34}
0<\|v(T)\|_{L^2}<(2C^r)^{-\frac{1}{p-1}}.
\end{align}

Thanks to (\ref{eq3}) and Lemma \ref{lem31}, we obtain for $t>T$,
\begin{align*}
\|v(t)\|_{L^2}^{2}\geq \|v(T)\|_{L^2}^{2}-\|v\|_{L^q([T,t]; L^r)}^r,
\end{align*}
which implies with (\ref{inq33}) and (\ref{inq34}) that
\begin{align*}
\|v(t)\|_{L^2}^{2}&\geq \|v(T)\|_{L^2}^{2}-C^r \|v(T)\|_{L^2}^r 
\\
&=\|v(T)\|_{L^2}^{2}\left(1-C^r \|v(T)\|_{L^2}^{p-1}\right)
\\
&\geq \frac{1}{2} \|v(T)\|_{L^2}^2 >0.
\end{align*}
This contradicts that $\displaystyle\lim_{t\to \infty} \| v(t)\|_{L^2}=0$.


\section{Proof of Theorem \ref{thm1}}

We decompose the Duhamel term of (\ref{CP2}) into the four terms in the same manner as in Takagi--Takizawa \cite{Takagi-Takizawa}:
\begin{align} \label{decomp1}
  &v(t)-e^{it \Delta}\phi= I(t)+\sum_{j=1}^{3} R_j(t),
\end{align}
where
\begin{align*}
&I(t)=i \int_{t}^{\infty}e^{-a(p-1)s} s^{-n(p-1)/2} M(t)D(t) N(\widehat{\phi})\,ds,   
\\
&R_1(t)=i \int_{t}^{\infty}e^{-a(p-1)s} s^{-n(p-1)/2} M(t)D(t) \SCR{F} \{M(t)M(-s)-1\} \SCR{F}^{-1} N(\widehat{\phi})\,ds,  
\\
&R_2(t)=i \int_{t}^{\infty}e^{-a(p-1)s}e^{i(t-s)\Delta}[N(e^{is\Delta}\phi)-N(M(s)D(s)\widehat{ \phi})]\,ds, 
\\
&R_3(t)=i \int_{t}^{\infty}e^{-a(p-1)s}e^{i(t-s)\Delta}\{N(v(s))-N(e^{is\Delta}\phi)\} \,ds. 
\end{align*}

We have
\begin{align*}
\|I(t)\|_{L^2}&=\int_{t}^{\infty} e^{-a(p-1)s} s^{-n(p-1)/2} ds \|N(\widehat{\phi})\|_{L^2},
\end{align*}
which, together with Lemma \ref{ElemLem}, leads to
\begin{align} \label{I}
A t^{-n(p-1)/2} e^{-a(p-1)t} \leq \|I(t)\|_{L^2} \leq B t^{-n(p-1)/2} e^{-a(p-1)t} 
\end{align}
for some constants $A, B>0$.
We see
\begin{align} \label{R1'}
\|R_1(t)\|_{L^2}&\leq \int_{t}^{\infty} e^{-a(p-1)s} s^{-n(p-1)/2}\|\left(M(t)M(-s)-1\right)\SCR{F}^{-1} N(\widehat{\phi})\|_{L^2} ds.
\end{align}
We note that $\SCR{F}^{-1} N(\widehat{\phi})\in L^2$ by the Hausdorff--Young inequality if $\phi\in L^{(2p)'}$.
Thus by the density there exists a sequence $\{f_m\}_{m=1}^{\infty}\subset \CAL{S}(\R^n)$ such that $\|f_m-\SCR{F}^{-1} N(\widehat{\phi})\|_{L^2}<1/m$. For $T\leq t\leq s$, we obtain
\begin{align*}
&\|\left(M(t)M(-s)-1\right)\SCR{F}^{-1} N(\widehat{\phi})\|_{L^2}
\\
&\leq \|\left(M(t)M(-s)-1\right)f_m\|_{L^2}+\|\left(M(t)M(-s)-1\right)(f_m-\SCR{F}^{-1} N(\widehat{\phi}))\|_{L^2}
\\
&\leq \left\|\frac{|x|^2}{2}\left(\frac{1}{t}-\frac{1}{s}\right)f_m\right\|_{L^2}+\frac{2}{m}
\\
&\leq \frac{1}{2T}\||\cdot|^2 f_m\|_{L^2}+\frac{2}{m}.
\end{align*}
Thus we also have
\begin{align*}
\limsup_{T\to\infty} \|\left(M(t)M(-s)-1\right)\SCR{F}^{-1} N(\widehat{\phi})\|_{L^2}\leq \frac{2}{m}
\end{align*}
for all $m\in \N$, which derives
\begin{align*}
\lim_{T\to\infty} \|\left(M(t)M(-s)-1\right)\SCR{F}^{-1} N(\widehat{\phi})\|_{L^2}=0.
\end{align*}
This and  (\ref{R1'}) show that there exists  $T=T(a,n,p,\phi)>0$ fulfilling
\begin{align} \label{R1}
\|R_1(t)\|_{L^2}&\leq \frac{A}{3} e^{-a(p-1)t} t^{-n(p-1)/2}.
\end{align}

By the H\"{o}lder inequality, the $L^q$-$L^{q'}$ estimate (\ref{LpLq}) and the Hausdorff--Young inequality, we can compute $R_2$ as follows:
\begin{align} \label{R2'}
\|R_2(t)\|_{L^2}&\leq \int_{t}^{\infty} e^{-a(p-1)s} \| N(e^{is\Delta}\phi)-N(M(s)D(s)\widehat{ \phi})\|_{L^2} ds  \notag
\\
&\lesssim \int_{t}^{\infty} e^{-a(p-1)s} \left\| \left(|e^{is\Delta}\phi|^{p-1}+| M(s)D(s)\widehat{ \phi}|^{p-1}\right)|e^{is\Delta}\phi- M(s)D(s)\widehat{ \phi}| \right\|_{L^2} ds  \notag
\\
&\lesssim \int_{t}^{\infty} e^{-a(p-1)s} \left\| \left(|e^{is\Delta}\phi|^{p-1}+| M(s)D(s)\widehat{ \phi}|^{p-1}\right) \right\|_{L^{2p'}} \|e^{is\Delta}\phi- M(s)D(s)\widehat{ \phi} \|_{L^{2p}} ds \notag
\\
&\lesssim \int_{t}^{\infty} e^{-a(p-1)s} \left(\|e^{is\Delta}\phi\|_{L^{2p}}^{p-1}+ \|M(s)D(s)\widehat{ \phi}\|_{L^{2p}}^{p-1}\right)  s^{-\frac{n(p-1)}{2p}} \|(M(s)-1) \phi \|_{L^{(2p)'}} ds \notag
\\
&\lesssim \int_{t}^{\infty} e^{-a(p-1)s} s^{-\frac{n(p-1)^2}{2p}} \|\phi \|_{L^{(2p)'}} s^{-\frac{n(p-1)}{2p}}  \|(M(s)-1)\phi \|_{L^{(2p)'}} ds \notag
\\
&\lesssim \int_{t}^{\infty} e^{-a(p-1)s} s^{-\frac{n(p-1)}{2}}  \|(M(s)-1)\phi \|_{L^{(2p)'}} ds \|\phi \|_{L^{(2p)'}}.
\end{align}
By Lemma \ref{ElemLem} and applying the fact $\displaystyle\lim_{s\to\infty}\|(M(s)-1)\phi\|_{L^{(2p)'}}=0$, which follows from the dominated convergence theorem, to (\ref{R2'}) there exists $T=T(a,n,p,\phi)>0$ such that
\begin{align} \label{R2}
\|R_2(t)\|_{L^2}&\leq \frac{A}{3} t^{-n(p-1)/2} e^{-a(p-1)t}
\end{align}
for $t>T$.
The Strichartz estimate and the H\"{o}lder inequality yield
\begin{align} \label{R3}
\|R_3(t)\|_{L^2}&\lesssim \| e^{-a(p-1)s} \left( N(v(s)-N(e^{is\Delta}\phi) \right)\|_{L^{q'}([t,\infty); L^{r'})} \notag
\\
&\lesssim   \|e^{-a(p-1)s} \left( |v(s)|^{p-1}+|e^{is\Delta}\phi|^{p-1} \right) (v(s)-e^{is\Delta}\phi)\|_{L^{q'}([t,\infty); L^{r'})} \notag
\\
&\lesssim  \left\|e^{-a(p-1)s} \| |v(s)|^{p-1}+|e^{is\Delta}\phi|^{p-1} \|_{L^{\frac{p+1}{p-1}}} \| v(s)-e^{is\Delta}\phi\|_{L^{r}} \right\|_{L^{q'}[t,\infty)}  \notag
\\
&\leq \left\| e^{-a(p-1)s}\left( \|v(s)\|_{L^r}^{p-1}+\|e^{is\Delta}\phi\|_{L^r}^{p-1} \right) \| v(s)-e^{is\Delta}\phi\|_{L^r}  \right\|_{L^{q'}[t,\infty)}  \notag
\\
&\leq \|e^{-a(p-1)s}\|_{L^{\alpha}[t,\infty)} \left\| \|v\|_{L^r}^{p-1}+\|e^{is\Delta}\phi\|_{L^r}^{p-1} \right\|_{L^{\frac{q}{p-1}}[t,\infty)} \| v(s)-e^{is\Delta}\phi\|_{L^r} \|_{L^{q}[t,\infty)}  \notag
\\
&\leq e^{-a(p-1)t} \left( \|v\|_{L^{q}([t,\infty); L^{r})}^{p-1}+\|e^{is\Delta}\phi\|_{L^{q}([t,\infty); L^{r})}^{p-1} \right) \| v(s)-e^{is\Delta}\phi\|_{L^{q}([t,\infty); L^{r})}  \notag
\\
&\leq e^{-a(p-1)t} \left( C^{p-1}+ \|\phi\|_{L^2}^{p-1} \right) \| v(s)-e^{is\Delta}\phi\|_{L^{q}([t,\infty); L^{r})}   \notag
\\
&\lesssim e^{-2a(p-1)t}.
\end{align}
Here we have used Proposition \ref{prop1} below.

Therefore we attain by (\ref{decomp1}) 
\begin{align*}
\|I(t)\|_{L^2}-\sum_{j=1,2,3} \|R_j(t)\|_{L^2} \leq \|v(t)-e^{it \Delta}\phi\|_{L^2}\leq \|I(t)\|_{L^2}+\sum_{j=1,2,3} \|R_j(t)\|_{L^2},
\end{align*}
which implies with (\ref{I}), (\ref{R1}), (\ref{R2}) and (\ref{R3}) that
\begin{align*}
C_1 t^{-n(p-1)/2}e^{-a(p-1)t} \leq \|v(t)-e^{it \Delta}\phi\|_{L^2}\leq C_2 t^{-n(p-1)/2}e^{-a(p-1)t}
\end{align*}
with some constant $C_2>C_1>0$.
This is the desired result.
\appendix
\section{Global well-posedness}
In this section we give the proof of Proposition \ref{prop1}.
\begin{proof}
\underline{STEP 1}  (Construction of global $L^2$-solution)

We first construct global solutions.
By Duhamel's formula, we have
\begin{align} \label{IE1}
 v(t)=e^{i(t-\tau)\Delta}v(\tau)
 -i\int_\tau^t e^{i(t-s)\Delta}e^{-a(p-1)s}N(v(s)) ds
\end{align}
for all $\tau\geq0$ and $t\geq\tau$.
We see that (\ref{IE1}) has a global solution $v \in C([0,\infty);L^2(\R^n))\cap L_{\mathrm{loc}}^q((0,\infty);L^r(\R^n))$ satisfying (\ref{eq3}) by an argument similar to the one in Y. Tsutsumi \cite{YTsutsumi2}.
Indeed, the only difference is the presence or
absence of the exponential decay $e^{-a(p-1)s}$ in (\ref{IE1}), which is harmless.

\vskip\baselineskip
\underline{STEP 2} ($v$ belongs to $L^{q}([0,\infty); L^r)$)

For $T<t$ and the global solution $v$ to (\ref{CP2}), the Duhamel formula and Lemma \ref{Lem2} give
\begin{align*}
\|v\|_{L^{q}([T,t]; L^r)}&\lesssim \|v(T)\|_{L^2}+\| e^{-a(1-1/p)s} \|_{L^{\beta}[T,t]}^{p-1}\|v\|_{L^{q}([T,t]; L^r)}^p
\\
&\lesssim \|v_0\|_{L^2}+e^{-a(p-1)T} \|v\|_{L^{q}([T,t]; L^r)}^p. 
\end{align*}
Noting that $e^{-a(p-1)T}\to0$ as $T\to\infty$, we have by a continuous argument $\|v\|_{L^{q}([T,t]; L^r)}\leq C\|v_0\|_{L^2}$ for all $t>T$ with some large $T=T(\|v_0\|_{L^2})>0$, which implies $\|v\|_{L^{q}([0,\infty); L^r)}<\infty$.

\vskip\baselineskip
\underline{STEP 3} (Existence of scattering state)

We prove existence of a scattering state for the global solution $v$ by Cook--Kuroda's method. For $t>t'$, the Strichartz estimate and Lemma \ref{Lem2} yield
\begin{align*}
\|e^{-it\Delta}v(t)-e^{-it'\Delta}v(t')\|_{L^2}&\lesssim \left\|\int_{t'}^{t}e^{-is\Delta}[ e^{-a(p-1)s}N(v(s))]ds \right\|_{L^2}
\\
&\lesssim e^{-a(p-1)t'} \|v\|_{L^{q}([t',t]; L^{r})}^{p}
\\
&\lesssim e^{-a(p-1)t'} \|v\|_{L^{q}([0,\infty); L^{r})}^{p}
\\
&\to 0
\end{align*}
as $t'\to \infty$, which implies that $\phi:=\displaystyle\lim_{t\to\infty} e^{-it\Delta}v(t)$ in $L^2$ exists. 

\vskip\baselineskip
\underline{STEP 4} (Scattering at exponential decay rate)

We have by the Duhamel formula, the Strichartz estimates and Lemma \ref{Lem2}
\begin{align*}
&\|v(t)-e^{it\Delta}\phi\|_{L^2}+ \|v(s)-e^{is\Delta}\phi\|_{L^q([t,\infty);L^r)}
\\
&=  \left\| \int_{t}^{\infty} e^{i(t-s)\Delta} e^{-a(p-1)s} N(v(s)) ds \right\|_{L^2}+ \left\| \int_{s}^{\infty} e^{i(s-s')\Delta} e^{-a(p-1)s'} N(v(s')) ds' \right\|_{L^{q}([t,\infty);L^{r})}
\\
&\lesssim e^{-a(p-1)t} \|v\|_{L^{q}([t,\infty);L^{r})}^p,
\end{align*}
which, combined with $\displaystyle\lim_{t\to \infty} \|v\|_{L^{q}([t,\infty);L^{r})}=0$, imply the conclusion. 
\end{proof}

\section{Final state problem}
\begin{Prop}\label{prop2}
For a given final state $\phi\in L^2 \cap L^{\frac{2p}{2p-1}}$, there exists a solution $v\in C([0,\infty); L^2)$ to (\ref{CP2}) with $\operatorname{Im}\lambda=0$ satisfying
\begin{align*}
C_1 \left(t^{-n/2}e^{-at}\right)^{p-1}\leq \|v(t)-e^{it\Delta}\phi\|_{L^2} \leq C_2 \left(t^{-n/2}e^{-at}\right)^{p-1}
\end{align*}
for all $t>T$ with some constants $C_1, C_2, T>0$.
\end{Prop}
\begin{proof}
We deal with the following integral equation from the Duhamel principle:
\begin{align*}
v(t)=e^{it\Delta}\phi+i\int_{t}^{\infty}e^{-a(p-1)s}e^{i(t-s)\Delta} N(v(s))ds,
\end{align*}
 which implies 
\begin{align*}
w(t)=i\int_{t}^{\infty}e^{-a(p-1)s}e^{i(t-s)\Delta} [N(w(s)+e^{is\Delta}\phi)-N(e^{is\Delta}\phi)+N(e^{is\Delta}\phi)]ds
\end{align*}
by the transformation: $w(t)=v(t)-e^{it\Delta}\phi$.
Let the right hand side be $\Gamma w(t)$ and 
\begin{align*}
X(T,M):=\{ w\in C([0,\infty); L^2) : \|w\|_{X_T}\leq M\}
\end{align*}
with the equipped norm:
\begin{align*}
\|w\|_{X_T}:=\sup_{t\in[T,\infty)} (t^{n/2}e^{at})^{p-1} \{ \|w(t)\|_{L^2}+\|w\|_{L^q([t,\infty); L^r)} \},
\end{align*}
where $(q,r):=\left(\frac{4(p+1)}{n(p-1)}, p+1\right)$ is the admissible pair. 

We can decompose as follows:
\begin{align*}
N(w(s)+e^{is\Delta}\phi)-N(e^{is\Delta}\phi)=A_1(s)+A_2(s)
\end{align*}
with
\begin{align*}
|A_1(s)|\leq C_p |e^{is\Delta}\phi|^{p-1} |w(s)| \hspace{3mm}\text{and} \hspace{3mm} |A_2(s)|\leq C_p |w(s)|^p
\end{align*}
since
\begin{align*}
| N(w(s)+e^{is\Delta}\phi)-N(e^{is\Delta}\phi)| \leq C_{p} |e^{is\Delta}\phi|^{p-1} |w(s)|+|w(s)|^p.
\end{align*}
The Strichartz estimates yield
\begin{align*}
&\|\Gamma w(t)\|_{L^2}+ \|\Gamma w(t)\|_{L^q([t,\infty); L^r)}
\\
&\lesssim \|e^{-a(p-1)s}A_1\|_{L^{q'} ([t,\infty);L^{p'})}+\|e^{-a(p-1)s}A_2\|_{L^{q'} ([t,\infty);L^{p'})}+\| e^{-a(p-1)s} N(e^{is\Delta} \phi) \|_{L^1([t,\infty); L^2)}
\\
&=: I+I\!I+I\!I\!I.
\end{align*}
We estimate I. By the H\"{o}lder inequality and (\ref{def1}) we derive
\begin{align*}
I&\lesssim \left\|e^{-a(p-1)s}\| |e^{is\Delta}\phi|^{p-1} w(s)\|_{L^{r'}} \right\|_{L^{q'}[t,\infty)}
\\
&\lesssim \left\|e^{-a(p-1)s}\| |e^{is\Delta}\phi|^{p-1}\|_{L^{\frac{r}{p-1}}} \|w(s)\|_{L^{r}} \right\|_{L^{q'}[t,\infty)}
\\
&=\left\|e^{-a(p-1)s}\| e^{is\Delta}\phi\|_{L^{r}}^{p-1} \|w(s)\|_{L^{r}} \right\|_{L^{q'}[t,\infty)}
\\
&\lesssim \left\| e^{-a(p-1)s}\| e^{is\Delta}\phi\|_{L^{r}}^{p-1} \right\|_{L^{\alpha}[t,\infty)}  \|w\|_{L^q([t,\infty); L^r)}
\\
&=\left\| e^{-as}\| e^{is\Delta}\phi\|_{L^{r}} \right\|_{L^{\alpha}[t,\infty)}^{p-1}  \|w\|_{L^q([t,\infty); L^r)}
\\
&\lesssim \| e^{-as}\|_{L^{\beta}[t,\infty)}^{p-1} \|e^{is\Delta}\phi\|_{L^q([t,\infty); L^r)}^{p-1} \|w\|_{L^q([t,\infty); L^r)}
\\
&\lesssim e^{-a(p-1)t} \|e^{is\Delta}\phi\|_{L^q([t,\infty); L^r)}^{p-1} \|w\|_{L^q([t,\infty); L^r)},
\end{align*}
which and the Strichartz estimate give
\begin{align} \label{est1}
I\lesssim \left(e^{-2at}t^{-n/2}\right)^{p-1} \|\phi\|_{L^2}^{p-1} \|w\|_{X_T}.
\end{align}

Similarly we have
\begin{align} \label{est2}
I\!I &\lesssim \left\| e^{-a(p-1)s} \| |w(s)|^p \|_{L^{r'}} \right\|_{L^{q'}[t,\infty)} \notag
\\
&\lesssim \| e^{-a(1-\frac{1}{p})s} \|w(s)\|_{L^r} \|_{L^{pq'}[t,\infty)}^{p} \notag
\\
&\lesssim \| e^{-a(1-\frac{1}{p})s}\|_{L^{\gamma}}^p \|w\|_{L^{q}[t,\infty); L^r)}^p \notag
\\
&\lesssim e^{-at(p^2-p+1)}t^{-n(p^2-p)/2} \|w\|_{X_T}^p,
\end{align}
where $\gamma$ satisfies $\frac{1}{pq'}=\frac{1}{\gamma}+\frac{1}{q}$.

The $L^q$-$L^{q'}$ estimate (\ref{LpLq}) yields
\begin{align} \label{est3}
I\!I\!I &\lesssim \| e^{-a(p-1)s} N(e^{is\Delta} \phi) \|_{L^1([t,\infty); L^2)} \notag
\\
&=\| e^{-a(p-1)s} \|e^{is\Delta} \phi\|_{L^{2p}}^p  \|_{L^1[t,\infty)}  \notag
\\
&\lesssim \| e^{-a(p-1)s} s^{-n(p-1)/2} \|_{L^1[t,\infty)} \|\phi\|_{L^{(2p)'}}^p  \notag
\\
&\lesssim e^{-a(p-1)t} t^{-n(p-1)/2} \|\phi\|_{L^{\frac{2p}{2p-1}}}^p
\end{align}
Therefore by (\ref{est1}), (\ref{est2}) and (\ref{est3}), we derive
\begin{align*}
\|\Gamma w\|_{X_T}\leq C e^{-aT} \left(\|\phi\|_{L^2}^{p-1} \|w\|_{X_T}+\|w\|_{X_T}^p\right) +C \|\phi\|_{L^{\frac{2p}{2p-1}}}^p 
\end{align*}
with some constant $C>0$.
Suppose that $w\in X(T, M)$.
By putting $M=2C \|\phi\|_{L^{(2p)'}}^{p}$ and choosing $T>0$ such as $Ce^{-aT}(\|\phi\|_{L^2}^{p-1}+M^{p-1})<1/2$, we have
\begin{align*}
\|\Gamma w\|_{X_T}< M/2 +M/2=M
\end{align*} 
Similarly, we obtain
\begin{align*}
\|\Gamma w_1 - \Gamma w_2\|_{X_T}< \frac{1}{2} \|w_1-w_2\|_{X_T}.
\end{align*}
Thus $\Gamma$ is a contraction map on $X(T, M)$.
Moreover we have the $L^2$-conservation law, which enables us to utilize the usual argument of construction for a local solution (see Y. Tsutsumi \cite{YTsutsumi2}).
Consequently,  we can extend the solution $v\in C([T,\infty); L^2)$ to $v\in C([0,\infty); L^2)$.
\end{proof}

\section{Alternative proof for Theorem \ref{thm2}}
Next we give an alternative proof for Theorem \ref{thm2}.   
\begin{proof}
Thanks to Proposition \ref{prop1}, the Cauchy problem (\ref{CP1}) has a global solution.
It suffices to prove a contradiction if $\displaystyle\lim_{t\to \infty} \|v(t)\|_{L^2}=0$ since the quantity $\|v(t)\|_{L^2}$ is monotone decreasing.  
We have by the Duhamel principle
\begin{align*}
v(t)=i  \int_{t}^{\infty} e^{i(t-s)\Delta}e^{-a(p-1)s} N(v(s)) ds.
\end{align*}
Applying the Strichartz estimates and Lemma \ref{Lem2} to this equality, we obtain 
\begin{align} \label{inq61}
\|v\|_{X(t)}\lesssim e^{-a(p-1)t} \|v\|_{X(t)}^p,
\end{align}
where we write $X(t):=L^{\infty}([t,\infty) ; L^2) \cap L^{q}([t,\infty); L^r)$ and $(q,r):=(\frac{4(p+1)}{n(p-1)}, p+1)$.
There exists a constant $T_0>1$ such that $\|v\|_{X(T_0)}=0$ by (\ref{inq61}) and a continuous argument. 
Noting that $\|\cdot\|_{X(t)}$ is monotone decreasing in $t$, we define
\begin{align*}
T_1:=\inf\{ t\in [0,\infty): \|v\|_{X(t)}=0\}.
\end{align*}
By the standard argument of local solutions, there exist a constant $\delta>0$ and a solution $v\in C([T_1-\delta, T_1]; L^2)$ to (\ref{CP1}) with the initial data $v_0=v(T_1)=0$, which derives
\begin{align*}
\| v(T_1-\delta) \|_{L^2}=0.
\end{align*}
This contradicts the definition of $T_1$.
\end{proof}

\textbf{Data Availability Statement} No data were generated in this study.
\vskip\baselineskip
\textbf{Conflict of interests} The authors declare that there is no conflict of interests
regarding the publication of this paper.
\vskip\baselineskip
\textbf{Acknowledgements} The first author is supported by KAKENHI Grant-in-Aid for Early-Career Scientists No. JP24K16947 and partially by KAKENHI Grant-in-Aid for Scientific Research (B) No. JP26K00612.

\vskip\baselineskip
\address{Department of Mathematics, Graduate School of Science, The University of Osaka, Toyonaka,  Osaka 560-0043, Japan}
\email{inui@math.sci.osaka-u.ac.jp}

\vskip\baselineskip

\address{Department of Mathematics, Faculty of Science, Tokyo University of Science, Kagurazaka 1-3, Shinjuku-ku, Tokyo 162-8601, Japan}
\email{1123703@ed.tus.ac.jp}

\end{document}